\documentclass[final, 3p, times]{elsarticle}%
\usepackage{graphicx}
\usepackage{xcolor}
\usepackage{colortbl}
\usepackage{epstopdf}
\usepackage{amssymb}
\usepackage{amsthm}
\usepackage{amsmath}
\usepackage{color}
\usepackage{dsfont}
\usepackage{mathtools}
\usepackage{extarrows}
\usepackage{hyperref}
\usepackage{bm}
\usepackage{caption}
\usepackage{float}
\usepackage{verbatim}
\usepackage{epsfig}
\usepackage{multirow}
\usepackage{subfig}
\usepackage[section]{placeins}
\usepackage{float}

\newtheorem{lem}{Lemma}[section]
\newtheorem{thm}{Theorem}[section]

\newtheorem{rmk}{Remark}[section]

\biboptions{compress}

\journal{XXX}

\begin{document}
\begin{frontmatter}
\title{ Asymptotic stability  of many numerical schemes for phase-field modeling}

\tnotetext[label1]{The research of Dongling Wang is supported in part by the National Natural Science Foundation of China under grants 12271463 and Outstanding youth fund of department of education of Hunan province under grants 22B0173.
 \\ Declarations of interest: none.
}

\author[XTU]{Pansheng Li}
\ead{lpsmath@smail.xtu.edu.cn;}
\author[XTU]{Dongling Wang\corref{mycorrespondingauthor}}
\ead{wdymath@xtu.edu.cn; ORCID 0000-0001-8509-2837}
\cortext[mycorrespondingauthor]{Corresponding author. }

\address[XTU]{Hunan Key Laboratory for Computation and Simulation in Science and Engineering, School of Mathematics and Computational Science, Xiangtan University, Xiangtan, Hunan 411105, China}

\begin{abstract}
The numerical stability of nonlinear equations has been a long-standing concern and there is no standard framework for analyzing long-term qualitative behavior. In the recent breakthrough work \cite{xu2023lack}, a rigorous numerical analysis was conducted on the numerical solution of a scalar ODE containing a cubic polynomial derived from the Allen-Cahn equation. It was found that only the implicit Euler method converge to the correct steady state for any given initial value $u_0$ under the unique solvability and energy stability. But all the other commonly used second-order numerical schemes exhibit sensitivity to initial conditions and may converge to an incorrect equilibrium state as $t_n\to\infty$. This indicates that energy stability may not be decisive for the long-term qualitative correctness of numerical solutions.

We found that using another fundamental property of the solution, namely monotonicity instead of energy stability, is sufficient to ensure that many common numerical schemes converge to the correct equilibrium state.
This leads us to introduce the critical step size constant $h^*=h^*(u_0,\epsilon)$ that ensures the monotonicity and unique solvability of the numerical solutions, where the scaling parameter $\epsilon \in(0,1)$.
For a given numerical method, if the initial value $u_0$ is given, no matter how large it is, we prove that $h^*>0$. As long as the actual simulation step $0<h<h^*$, the numerical solution preserves monotonicity and converges to the correct equilibrium state. On the other hand, we prove that the implicit Euler scheme $h^*=h^*(\epsilon)$, which is independent of $u_0$ and only depends on $\epsilon$. Hence regardless of the initial value taken, the simulation can be guaranteed to be correct when $h<h^*$. But for various other numerical methods,  no mater how small the step size $h$ is in advance, there will always be initial values that cause simulation errors. In fact, for these numerical methods, we prove that  $\inf_{u_0\in \mathbb{R}}h^*(u_0,\epsilon)=0$. Various numerical experiments are used to confirm the theoretical analysis.

\end{abstract}


\end{frontmatter}

\section{introduction}

The phase-field model is an essential tool for describing phase transitions and interface dynamics in materials science. The origins of this model can be traced back to the pioneering work of van der Waals \cite{van1979thermodynamic}. Among the various phase-field models, the Allen-Cahn equation, Cahn-Hilliard equation and the Molecular Beam Epitaxy (MBE) equation are some of the most commonly used.
In this work, we will focus specifically on one of the most widely studied phase-field models: the Allen-Cahn equation \cite{allen1979microscopic}
\begin{equation}\label{A1}
\begin{aligned}
u_t-\Delta u+\frac{1}{\epsilon^2} f(u) & =0 \quad \text { in } \Omega_T=\Omega \times(0, T), \\
\frac{\partial u}{\partial n} & =0 \quad \text { on } \partial \Omega_T=\partial \Omega \times(0, T), \\
\left.u\right|_{t=0} & =u_0.
\end{aligned}
\end{equation}
It is well-known that this model can be viewed as the $L^2$ gradient flow of the free
energy functional
\begin{equation}
E(u)=\int_\Omega\frac12|\nabla u|^2+\frac1{\epsilon^2}F(u)dx
\end{equation}
and satisfies the energy dissipation law:
\begin{equation}\label{eq:energy}
\frac d{dt}E(u(t))=-\int_\Omega|-\Delta u+\frac1{\epsilon^2}f(u)|^2dx=-\int_\Omega|u_t|^2dx \leq0,
\end{equation}
where $f(u)=u^3-u=F^{\prime}(u)$ and $F(u)=\frac14(u^2-1)^2$, the scaling coefficient $\epsilon \in(0,1)$ represents the interfacial width.

The numerical simulation of phase field equations has a wide range of practical applications. However, designing efficient and stable numerical schemes has always been a significant challenge, primarily due to the following reasons:
(1) Numerical algorithms need to accurately capture the dynamic information of phase transitions while ensuring system stability during long-term simulations;
(2) Phase field models should satisfy specific physical properties, such as mass conservation and energy dissipation;
(3) The strong nonlinearity present in phase field problems makes constructing efficient high-order methods particularly challenging.

In the early studies of numerical simulations of phase field equations, fully explicit and implicit methods dominated. Fully explicit methods include, for example, the Explicit Euler scheme \cite{jeong2016comparison,feng2003numerical} and higher-order explicit Runge-Kutta methods \cite{zhang2021numerical}. Fully implicit methods include, for example, the Implicit Euler scheme \cite{jeong2016comparison}, Crank-Nicolson \cite{crank1947practical,hou2017numerical,zhang2009numerical}, the modified Crank-Nicolson  scheme \cite{condette2011spectral,du1991numerical,shen2010numerical}, and Diagonally Implicit Runge-Kutta methods \cite{zhang2021preserving}. However, both fully explicit and fully implicit methods are inefficient when solving practical problems, especially for long-term evolution.

To address these challenges, partially implicit schemes were developed based on the principle of energy stability.
That means similar to continuous energy dissipation \eqref{eq:energy}, the discrete energy does not increase, i.e., $E(u_{n+1})\leq E(u_{n})$ for the numerical solutions $\{u_{n}\}_{n=0}^{\infty}$.
These include, for example, the convex splitting scheme \cite{eyre1998unconditionally,graser2013time,guan2014second,guillen2014second}, stabilized methods \cite{feng2013stabilized,he2007large,xu2006stability}, the invariant energy quadratization method \cite{yang2016linear} and the scalar auxiliary variable method \cite{shen2019new,shen2018scalar}. These are all methods that ensure either energy stability or modified energy stability $\widetilde{E}(u_{n+1})\leq \widetilde{E}(u_{n})$.

However, recent literature \cite{xu2019stability} indicates that these partially implicit schemes may achieve energy stability at the cost of sacrificing accuracy.
Since the numerical stability of nonlinear equations has been a long-standing concern and lack of a general theoretical analysis framework, especially for long-term qualitative behavior. The recent literature \cite{xu2023lack} has discovered some novel stability phenomena while examining the follow scalar nonlinear ODE associated with the Allen-Cahn equation

\begin{equation}\label{1.1}
u_t+\frac1{\epsilon^2}f(u)=0 \text { with initial value } u(0)=u_0.
\end{equation}
The ODE $(\ref{1.1})$ can be solved exactly  \cite{stuart1998dynamical}
\begin{equation}
u(t)=\frac{u_0}{\sqrt{e^{-\frac2{\epsilon^2}t}+u_0^2\left(1-e^{-\frac2{\epsilon^2}t}\right)}}.
\end{equation}
\begin{figure}[H]
	\centering
	\subfloat{\label{Fig-11}
		\centering
		\includegraphics[width=0.4\textwidth]{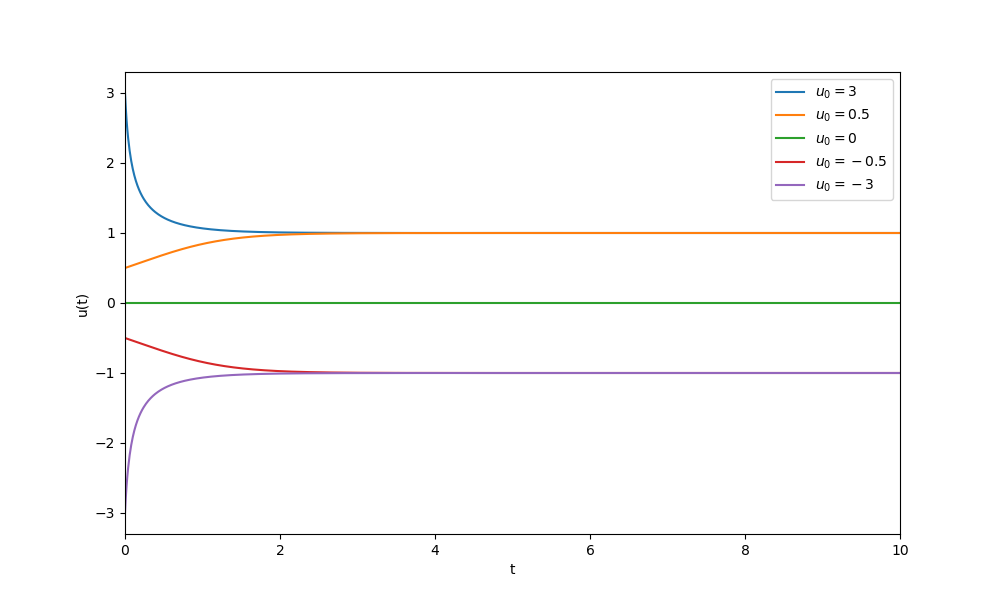}}
	\hspace{0.1cm}
	\caption{Exact solution $u(t)$ for different initial values}
    \label{FigureA}
\end{figure}

As $t\rightarrow +\infty$, the solution
exponentially converges to sign$(u_0)$,
which corresponds to three steady state solutions $\pm 1$ and $0$, depending on the sign of $u_0$, as shown in the Figure \ref{FigureA}. Following \cite{xu2023lack}, we say a numerical solution $u_n$ generated by some given numerical schemes approximating $u(t_n)$ at $t_n=nh$ with step size $h>0$, converges to the correct steady state solution of
\eqref{1.1} if
\begin{equation}\label{eq:xx}
\lim_{n\to\infty}u_n=\mathrm{sign}(u_0).
\end{equation}

Similar to \eqref{eq:energy}, the following energy law holds for the ODE \eqref{1.1}
\begin{equation}\label{eq:ode-energy}
\frac d{dt}E(u(t))=-|u_t|^2 \leq0,
\end{equation}
where $E(u)=\frac{1}{\epsilon^{2}}F(u).$  From now on, we will always use this definition for the
energy $E(\cdot)$ when discussing numerical schemes for this ODE model \eqref{eq:energy}.
As pointed out in \cite{xu2023lack},  any solution of the ODE $(\ref{1.1})$ is also a solution of the Allan-Chan equation $(\ref{A1})$.
As a result, any dynamical issues of a numerical scheme for the ODE $(\ref{1.1})$, such
as stability, accuracy and convergence to the correct steady state solution, should
also exist in general for the PDE model $(\ref{A1})$.

In this paper, we consider two class of the most commonly used numerical schemes for
phase-field models, namely the first-order and second-order
schemes. The first-order schemes include the explicit and implicit Euler schemes while the
second-order schemes include the Crank-Nicolson  scheme, the modified Crank-Nicolson scheme, the convex splitting scheme based on the modified Crank-Nicolson scheme and the implicit midpoint scheme \cite{wang2016implicit, ascher1999midpoint}.  It is well-known that the explicit Euler method is conditionally stable and often necessitates very small time steps, particularly for stiff problems. On the other hand, the implicit Euler, Crank-Nicolson, and implicit midpoint schemes are unconditionally stable for linear test equation and are well-suited for dealing with stiff problems \cite{Hairer}. 
However, they encounter the challenge of having to solve a nonlinear system at each time step, where the existence and uniqueness of the solution are concerns.
Recent literature \cite{ham2023stability} indicates that when employing the explicit Euler method to solve the Allen-Cahn equation \eqref{A1}, under certain time step constraints, it satisfies the maximum principle and energy stability.
Furthermore, the work \cite{xu2023lack} shows that for equation \eqref{1.1}, the Implicit Euler, Crank-Nicolson, and modified Crank-Nicolson schemes are conditionally uniquely solvable, and exhibit conditional energy stability or modified energy stability. It also states that the convex splitting scheme based on the modified Crank-Nicolson scheme is unconditionally uniquely solvable and unconditionally energy stable. However, regarding the standard implicit midpoint scheme, the energy stability for equations \eqref{A1} or \eqref{1.1} remains unknown. Energy stability is often achieved by second-order approximation of the nonlinear term, effectively transforming it into the modified Crank-Nicolson scheme \cite{li2019unconditionally, du1991numerical}.

The study in \cite{xu2019stability, xu2023lack} indicate that several second-order implicit schemes mentioned above all exhibit a common issue:
even if the numerical method is uniquely solvable and energy stable, it is still possible to converge to the wrong equilibrium state or experience oscillations.
Specifically,  for any $h>0$, there exists an initial condition $u_0$ with $|u_0|>1$ such that the numerical solution converges to wrong equilibrium state.
Moreover, for $|u_0|\leq 1$, all second-order schemes studied in \cite{xu2023lack} converge to the correct steady state solution but may experience oscillations if the time step size is not sufficiently small. However, the implicit Euler method is quite different from all other methods and it can converges to correct equilibrium states regardless of initial conditions. Therefore, the energy stability of numerical solutions may not be a decisive factor for long-term computation of phase field models.

Following \cite{xu2023lack}, we hope to find another property of the numerical schemes to replace energy stability that enables their numerical solution can preserve the qualitative properties of the original equation, that is, to converge to the correct equilibrium state without numerical oscillations.
 According to the classical theory of ordinary differential equations, the solutions of the scalar ODE $u'(t)=f(u)$ keep their monotonicity, due to the facts that the solution curves never cross the zeros of $f$ and hence $f(u)$ has a definite sign.
This observation inspires us to use the monotonicity of numerical solutions instead of energy stability, leading to the critical step $h^*$ defined below.
\begin{equation}\label{eq:hstart}
 h^* = \sup \Big\{ h=h(u_{0}, \epsilon) \mid  \text{Keep the unique solvability and monotonicity of numerical solutions} \Big\}.
\end{equation}

For a given numerical method, if the initial value $u_0$ is given, no matter how large it is, we prove that $h^*>0$. As long as the actual simulation step $h\in(0, h^*]$, the numerical solution preserves monotonicity and converges to the correct equilibrium state. On the other hand, we prove that for implicit Euler scheme $h^*=h(\epsilon)$, which is independent of $u_0$ and only depends on $\epsilon$ coming from the solvability condition. Hence regardless of the initial value taken, the simulation can be guaranteed to be correct when $h<h^*$. But for various other numerical methods, when the initial value $|u_0|> 1$,  $h^*$ depends not only on $\epsilon$~ but also on $u_0$.
Moreover, we find that $h^*$ is always less than step size limitation imposed by the unique solvability condition, so no mater how small the step size $h$ is in advance, there will always be initial values that cause simulation errors. In fact, for these numerical methods, we prove that  $\inf_{u_0\in \mathbb{R}}h(u_0,\epsilon)=0$.

The rest of the paper is organized as follows. In Section \ref{sec:firstorder}, we engage in a detailed discussion of the  explicit and implicit Euler schemes. In Section \ref{sec:secorder}, we discuss four second-order schemes. We first discussed the Crank-Nicoslon scheme in detail, and the other three numerical schemes can be discussed similarly, so we only provide the general ideas and main results.
Numerical results are presented and discussed in Section \ref{sec:num}.  Finally, some concluding remarks are given in Section \ref{sec:con}

\section{Explicit and implicit Euler schemes}
\label{sec:firstorder}
In this section, we discuss the properties of the explicit Euler (EE) and implicit Euler (IE) schemes. Firstly, for the EE scheme, we demonstrate that for any time step size
$h > 0$, it is always possible to find an initial condition $u_0$ such that the scheme converges to an incorrect steady state solution. Next, we present the critical value condition of time step that ensures the numerical solution monotonically converges to the correct equilibrium state.
Subsequently, we demonstrated that the numerical solution obtained by the explicit method retains energy stability when the time step size is within the critical value constraints. In the end, we found that the unique solvability condition for implicit Euler implies the monotonicity of the numerical solution.

\subsection{Explicit Euler}
The explicit Euler  scheme for the ODE  $(\ref{1.1})$ is given by
\begin{equation}\label{eq:EE}
\frac{u_{n}-u_{n-1}}{h}+\frac1{\epsilon^2}\big(u_{n-1}^3-u_{n-1}\big)=0, \quad n\geq 1.
\end{equation}
The work \cite{ham2023stability} shows that the explicit Euler method preserves the maximum principle and energy stability when the time step is sufficiently small. However, we now show that no mater how small the step size $h$ is in advance, there always exist some initial values that cause simulation errors.

\begin{lem}\label{lem:EE}
 For any time step $h>0$, there exists an initial condition $u_0$ such that the explicit Euler scheme converges to the incorrect steady state solution.
\end{lem}

\begin{proof}
The basic idea is that we first assume $u_1=-1$ and then solve $u_0$ in reverse from equation \eqref{eq:EE}.
Let
\[
g(v)=v-\frac{h}{\epsilon^2}(v^3-v)+1.
\]
It is easy to see $g(0)=1>0$ and $g(v)\rightarrow -\infty$ as $v\rightarrow +\infty$. Thus, there exists $u_0 > 0$ such that $g(u_0)=0$ by continuity of $g$.
This implies $u_1 = - 1$ solves \eqref{eq:EE} for $n = 1$. Furthermore, it follows from \eqref{eq:EE} that $u_{n}=-1$ for all $n\geq1$ as long as $u_{1}=-1$.
Therefore,  $u_n = -1\neq \mathrm{sign}(u_0)$ for any $n\geq 1.$

In the same way, we can show that there exist $u_0 < 0 $ such that $u_1 = 1$, which is the solution to \eqref{eq:EE} for $n = 1$. We then also have
$u_n = 1 \neq \mathrm{sign}(u_0)$ for all $n \geq 1$.
\end{proof}

We see from Lemma \ref{lem:EE} that when the initial value $|u_0|>1$, the time step size restrictive requirements such that the numerical solutions of the explicit Euler scheme converge to the correct equilibrium state depends on the initial value $u_0$. The following theorem shows this dependence and establish the optimal upper bound that allows numerical solutions to preserve monotonicity.

\begin{thm}\label{thm:EE}
Given an initial condition $u_0$, the explicit Euler scheme \eqref{eq:EE} for the has the following properties.

(i) If $u_0\in \{0,\pm1\}$, then $u_n = \mathrm{sign}(u_0)$ for all $n\geq1$ and for any $h$ and $\epsilon$;

(ii) If $u_0\notin \{0,\pm1\}$, then there exists a constant $h^*=h^*(u_0, \epsilon)>0$ such that for any $h\in(0, h^*]$, $u_n$ monotonically converge to correct equilibrium state $sign (u_0) $ as $n\rightarrow\infty$. Specifically, we have $h^*=\frac{\epsilon^2}{u_0^2+|u_0|}$ for $|u_0|>1$ and $h^*=\frac{\epsilon^2}{2}$ for $0< |u_0|< 1$ respectively.
 \end{thm}

\begin{proof}
For the first-order nonlinear difference equation \eqref{eq:EE}, we can determine the equilibrium points by solving the equation
$\frac{1}{\epsilon^2}\left({u^{*}}^3-u^{*}\right)=0$, thereby obtaining the same equilibrium states as those in the continuity equation \eqref{1.1}, namely, $\pm1$ and 0.
We will only consider the case where $u_0\geq0$, as the other case with $u_0<0$ can be demonstrated analogously.
The explicit Euler scheme \eqref{eq:EE} can be reformulated as:
\begin{equation}\label{eq:EE2}
u_n-u_{n-1}+\frac{h}{\epsilon^2}(u_{n-1}^3-u_{n-1})=0.
\end{equation}

Evidently, for any specified initial value $u_0$, there exists a unique solution $\{u_n\}_{n=1}^{\infty}$ to \eqref{eq:EE2} for any values of $h$ and $\epsilon$.
Subsequently, we examine the monotonicity of the numerical solutions.

Firstly, it is straightforward to deduce from \eqref{eq:EE2} that if $u_0=0$ or 1, then $u_n = \mathrm{sign}(u_0)$ for all $n\geq1.$

Secondly, we will analyze the monotonicity of $u_n$  for $u_0>1$.
Consider the auxiliary function derived from \eqref{eq:EE2}:
\begin{equation}\label{eq:EE3}
\psi (x)=x-\gamma+\frac{h}{\epsilon^2}(\gamma^3-\gamma).
\end{equation}
By considering the auxiliary function \eqref{eq:EE3} with $\gamma>1$, it is straightforward to observe that $\psi'(x)=1>0$. Assuming $ h \in (0, \frac{\epsilon^2}{\gamma^2 + \gamma}] $, then we have $\psi(1) \leq 0 $ and $\psi(x) \rightarrow +\infty $ as $x \rightarrow +\infty$. Consequently, there exists a unique $x^* \geq 1$ such that $\psi(x^*) = 0$. This implies that if $h \in (0, \frac{\epsilon^2}{u_{n-1}^2 + u_{n-1}}] $, then $u_n \geq 1$ for all $n \geq 1$. Furthermore, from \eqref{eq:EE2}, we derive that $u_n - u_{n-1} = -\frac{h}{\epsilon^2} \left((u_{n-1})^3 - u_{n-1}\right) \leq 0$. Thus, we obtain that
\[
\frac{\epsilon^2}{u_{n-1}^2 + u_{n-1}} \leq \frac{\epsilon^2}{u_{n}^2 + u_{n}}.
\]
Therefore, when \( 0 < h \leq \min_{n \in \mathbf{N}} \left\{\frac{\epsilon^2}{u_n^2 + u_n}\right\} = \frac{\epsilon^2}{u_0^2 + u_0} \), $\{u_n\}$ is monotonically decreasing and has a lower bound of $1$ for all \( n \geq 1 \). Taking the limit with respect to $n$ in  \eqref{eq:EE2} and incorporating the condition that \( u_n \geq 1 \) for all \( n \geq 1 \), we deduce that \( u^* = 1 \). Hence, when \( h \in (0, \frac{\epsilon^2}{u_0^2 + u_0}] \), $\{u_n\}$ monotonically decreases and converges to \( 1 = \mathrm{sign}(u_0) \) as \( n \rightarrow \infty \).

Thirdly, assume that \( 0 < u_0 < 1 \). By employing the auxiliary function \eqref{eq:EE3} with \( 0 < \gamma < 1 \), we can infer that for 
\( h \in (0, \frac{\epsilon^2}{u_{n-1}^2 + u_{n-1}}] \), it holds that \( 0 < u_n \leq 1 \) for all \( n \geq 1 \). This is follows from the fact \( \psi(0) < 0 \) and \( \psi(1) \geq 0 \), and it implies that
\[
u_n - u_{n-1} = -\frac{h}{\epsilon^2} \left((u_{n-1})^3 - u_{n-1}\right) \geq 0.
\]
 Consequently, we have
\(
\frac{\epsilon^2}{u_{n-1}^2 + u_{n-1}} \geq \frac{\epsilon^2}{u_{n}^2 + u_{n}}.
\)
Therefore, when \( 0 < h \leq \inf_{n \in \mathbf{N}} \left\{\frac{\epsilon^2}{u_n^2 + u_n}\right\} = \frac{\epsilon^2}{2} \), $\{u_n\}$ is monotonically increasing and has a upper bound of $1$ for all \( n \geq 1 \). Ultimately, we find that when \( h \in \left(0, \frac{\epsilon^2}{2}\right] \), $\{u_n\}$ monotonically increases and converges to \( 1 = \mathrm{sign}(u_0) \) as \( n \rightarrow \infty \).
\end{proof}


We prove that under the conditions of Theorem \ref{thm:EE}, for any initial value $u_0$, the numerical solution satisfies discrete energy stability.

\begin{thm}\label{thm:EE2}
      Given $u_0$, when $h\in (0,h^*]$, the EE scheme $\eqref{eq:EE}$ is energy stability, i.e., $E(u_n)\leq E(u_{n-1})$.
\end{thm}

\begin{proof}
Multiplying \eqref{eq:EE} by $\frac{h}{4\epsilon^2}(u_n^3+u_{n-1}^3)$ gives
\[
\frac{1}{4\epsilon^2}(u_n^4-u_{n-1}^4+u_nu_{n-1}^3-u_n^3u_{n-1})+\frac{h}{4\epsilon^4}(u_n^3+u_{n-1}^3)\big(u_{n-1}^3-u_{n-1}\big)=0.
\]
Hence, we have
\begin{equation}\label{eq4:EE}
E(u_n)-E(u_{n-1})+\frac{1}{4\epsilon^2}(2-u_nu_{n-1})(u_n^2-u_{n-1}^2)+\frac{h}{4\epsilon^4}(u_n^3+u_{n-1}^3)\big(u_{n-1}^3-u_{n-1}\big)=0,
\end{equation}
where $E(v)=\frac{1}{4\epsilon^2}(v^2-1)^2$. Furthermore, we have $u_n-u_{n-1}=-\frac{h}{\epsilon^2}(u_{n-1}^3-u_{n-1})$ from \eqref{eq:EE}. Substituting this into \eqref{eq4:EE} yields that 
$
E(u_n)-E(u_{n-1})+\frac{h}{4\epsilon^4}(u_nu_{n-1}-2)(u_n+u_{n-1})(u_{n-1}^3-u_{n-1})+\frac{h}{4\epsilon^4}(u_n^3+u_{n-1}^3)(u_{n-1}^3-u_{n-1})=0,
$
which can be written as
\begin{equation}\label{eq6:EE}
E(u_n)-E(u_{n-1})+\frac{h}{4\epsilon^4}(u_{n-1}^3-u_{n-1})(u_n+u_{n-1})\left(u_n^2+u_{n-1}^2-2\right)=0.
\end{equation}
Since $h\in (0,h^*]$, it follows from Theorem \ref{thm:EE} that $\{u_n\}$ maintains monotonicity and does not cross the equilibrium state $u*=1$, we have
\begin{equation}\label{eq7:EE}
\frac{h}{4\epsilon^4}(u_{n-1}^3-u_{n-1})(u_n+u_{n-1})\left(u_n^2+u_{n-1}^2-2\right)>0,
\end{equation}
Therefore, it follows  from \eqref{eq6:EE} that $E(u_n)\leq E(u_{n-1})$.
%
\end{proof}

\subsection{Implicit Euler}
The implicit Euler (IE) scheme for the ODE  $(\ref{1.1})$ is given by

\begin{eqnarray}\label{eq:IE}
\frac{u_n-u_{n-1}}{h}+\frac1{\epsilon^2}\big(u_n^3-u_n\big)=0, \quad n\geq 1.
\end{eqnarray}
Recent literatures \cite{xu2023lack,cheng2021asymptotic,xu2019stability,hao2024stability} has  indicated that the implicit Euler method is conditionally energy stable and has very superior performance compared to other methods commonly used in phase field equations. We now analyzse the monotonicity of the IE scheme $(\ref{eq:IE})$.

\begin{thm}
      Given an initial condition $u_0$, there exists $h^*=\epsilon^2$ such that for any $h\in(0, h^*]$, the IE scheme $(\ref{eq:IE})$ for the
      ODE $(\ref{1.1})$ has the following properties.

(i) If $u_0\in \{0,\pm1\}$, then $u_n = \operatorname{sign}(u_0)$  for all $n\geq1.$

(ii) If $u_0\notin \{0,\pm1\}$, then $u_n \rightarrow \operatorname{sign}(u_0) $ monotonically as $n\rightarrow+\infty$.

 \end{thm}

\begin{proof}

Similar to Theorem \eqref{thm:EE}, it is easy to obtain that equation \eqref{eq:IE}  has the same equilibrium points as the continuous equation \eqref{1.1}, i.e., $\pm1$ and 0.

Firstly, when $h\leq \epsilon^2$, IE is uniquely solvable \cite{xu2023lack}. It is easy to see from \eqref{eq:IE}  that if $u_0\in \{0,\pm1\}$, then $u_n=\mathrm{sign}(u_0)$ for all $n\geq1$.

Next, we will discuss the monotonicity of $u_n$  for $u_0>1$. Consider the auxiliary function
\begin{equation}
\psi(x)=\frac{h}{\epsilon^2}x^3+\left(1-\frac{h}{\epsilon^2}\right)x+\gamma,
\end{equation}
where $\gamma<-1$. Since $h\leq \epsilon^2$, then $\psi'(x)> 0$ for any $x\neq 0$. However, $\psi(1)=1+\gamma<0$, and $\psi(x)\rightarrow+\infty $ as $x\rightarrow+\infty $. Thus, there exists unique $x^*>1$ such that $\psi(x^*)=0$, which implies $1<u_n<+\infty$ for all $ n\in [1,+\infty)$.  Therefore, it follows from equation \eqref{eq:IE}  that $u_n-u_{n-1}=\frac{h}{\epsilon^2}\big((u_n)^3-u_n\big)<0$. Now, taking the limit with respect to $n$ in \eqref{eq:IE} and combining it with the condition that $u_n\geq1$ for all $n\geq1$, then we know that $u^*=1$. Consequently, when $h \leq \epsilon^2$, $\{u_n\}$ monotonically decreases and converges to $1=\text{sign}(u_0)$ as $n\to\infty$. For $u_0<1$, similar arguments can yield the required results.
\end{proof}


We observe that the limitation on step size for the implicit Euler method only comes from the unique solvability. Under this condition, the numerical solution remains monotonic and converges to the correct equilibrium state. This is completely different from the other second-order numerical formats that will be discussed below..

\section{Second order scheme}
\label{sec:secorder}

In this section, we study four second-order numerical methods. We first discussed the Crank-Nicoslon scheme in detail, and the other three numerical schemes can be discussed similarly, so we only provide the general ideas and main results.

\subsection{Crank-Nicolson scheme}
The second-order CN scheme for the ODE  $(\ref{1.1})$ is given by

\begin{eqnarray}\label{eq:CN}
\frac{u_n-u_{n-1}}{h}+\frac{1}{2 \epsilon^2} \left(u_n^3-u_n\right)+\frac{1}{2 \epsilon^2} \left(u_{n-1}^3-u_{n-1}\right)=0, \quad n\geq 1.
\end{eqnarray}

\begin{lem}\label{lem:CN1}
(\cite{xu2023lack})  When $h\leq 2\epsilon^2$, the CN scheme \eqref{eq:CN} is uniquely solvable and satisfies the modified energy stability
\begin{equation}\label{eq:CN2}
E(u_n)+\frac{h}{4\epsilon^4}[u_n^3-u_n]^2\leq E(u_{n-1})+\frac{h}{4\epsilon^4}[u_{n-1}^3-u_{n-1}]^2.
\end{equation}
However, for any step size $h>0$, there always exists an initial condition $u_0$ with $\left|u_0\right|>1$,
such that the CN scheme converges to the wrong steady state solution.
 \end{lem}

Lemma $\ref{lem:CN1}$ shows that the unique solvability and modified energy stability of the CN scheme do not necessarily ensure that the numerical solution  converge to the correct steady-state solution and reveals that the limitation on the time step size is intrinsically connected to the initial value $u_0$. The following Theorem \ref{thm:CN} elucidates this dependency.

\begin{thm}\label{thm:CN}
      Given an initial condition $u_0$, the CN scheme \eqref{eq:CN}  for the
      ODE $(\ref{1.1})$ has the following properties.

(i) If $u_0\in \{0,\pm1\}$, when $h\leq2\epsilon^2$, then $u_n = \operatorname{sign}(u_0)$  for all $n\geq1.$

(ii) If $u_0\notin \{0,\pm1\}$, then there exists a constant $h^*=h^*(u_0, \epsilon)>0$ such that for any $h\in(0, h^*]$, $\{u_n\}$ monotonically converge to correct equilibrium state $\mathrm{sign}(u_0) $ as $n\rightarrow\infty$. Specifically, we have $h^*=\frac{2\epsilon^2}{u_0^2+|u_0|}$ for $|u_0|>1$ and $h^*=\epsilon^2$ for $0< |u_0|< 1$ respectively. Hence, $\inf_{u_0\in \mathbb{R}}h^*(u_0,\epsilon)=0$.
\end{thm}

\begin{proof}

Similar to Theorem \eqref{thm:EE}, it is easy to obtain that equation \eqref{eq:CN}  has the same equilibrium points as the continuous equation \eqref{1.1}, i.e., $\pm1$ and $0$.
We only consider that $u_0\geq0$ and the other case $u_0<0$ can be proved in a similar way.
The CN scheme \eqref{eq:CN} can be written in the following form:
\begin{equation}\label{eq:CN2}
\frac{h}{2\epsilon^2}u_n^3+\left( 1-\frac{h}{2\epsilon^2}\right)u_n+\frac{h}{2\epsilon^2}(u_{n-1}^3-u_{n-1})-u_{n-1} =0.
\end{equation}
Next, we consider the monotonicity of numerical solutions.

When $h\leq2\epsilon^2$, CN scheme is uniquely solvable, it is easy to see from \eqref{eq:CN2}  that if $u_0=0$ or $1$, then $u_n = \mathrm{sign}(u_0)$  for all $n\geq1$.

We now analysis the monotonicity of $\{u_n\}$ for $u_0>1$. 
Consider the auxiliary function derived from \eqref{eq:CN2} as
\begin{equation}\label{eq:CN3}
\psi (x)=\frac{h}{2\epsilon^2}x^3+\left( 1-\frac{h}{2\epsilon^2}\right)x+\frac{h}{2\epsilon^2} (\gamma^3-\gamma)-\gamma,
\end{equation}
where $\gamma>1$. Assume $h\in\left(0, \frac{2\epsilon^2}{\gamma^2+\gamma}\right]$, it is easy to see that $\psi'(x)>0$. Moreover, we observe that $\psi (1)\leq 0$ and $\psi(x)\rightarrow+\infty $ as $x\rightarrow+\infty $, then there exists a unique $x^*\geq1$ such that $\psi(x^*)=0$. This implies that when $h\in\left(0, \frac{2\epsilon^2}{u_{n-1}^2+u_{n-1}}\right]$, then $u_n\geq1$ for all $n\geq1$.
It follows from \eqref{eq:CN}
$u_n-u_{n-1}=-\frac{h}{2 \epsilon^2}\left[ \left(u_n^3-u_n\right)+ \left(u_{n-1}^3-u_{n-1}\right)\right]\leq0$  and 
$\frac{2\epsilon^2}{u_{n-1}^2+u_{n-1}}\leq \frac{2\epsilon^2}{u_{n}^2+u_{n}}.$
Therefore, when $0<h\leq \min\limits_{n\in \mathbf{N}}\left\{\frac{2\epsilon^2}{u_n^2+u_n}\right\}=\frac{2\epsilon^2}{u_0^2+u_0}$, $\{u_n\}$ is monotonically decreasing and has a lower bound $1$ for all $n\geq1$. Now, taking the limit with respect to $n$ in \eqref{eq:CN} and combining it with the condition that $u_n\geq1$ for all $n\geq1$, we can get that when $h\in\left(0, \frac{2\epsilon^2}{u_0^2+u_0}\right]$, $\{u_n\}$ monotonically decreases and converges to $u^*=1=\mathrm{sign} \left(u_0\right)$ as $n \rightarrow \infty$.

Thirdly, consider that $0<u_0<1$. Taking the auxiliary function \eqref{eq:CN3} with $0<\gamma<1$, we can deduce that when \( h \in \left(0, \frac{2\epsilon^2}{u_{n-1}^2 + u_{n-1}}\right] \),  \( 0 < u_n \leq 1 \) for all \( n \geq 1 \).  This and equation \eqref{eq:CN} imply that
$u_n-u_{n-1}=-\frac{h}{2 \epsilon^2}\left[ \left(u_n^3-u_n\right)+ \left(u_{n-1}^3-u_{n-1}\right)\right]\geq0$.
 Consequently, we have
$
\frac{2\epsilon^2}{u_{n-1}^2+u_{n-1}}\geq \frac{2\epsilon^2}{u_{n}^2+u_{n}}.
$
Therefore, when $0<h\leq \inf\limits_{n\in \mathbf{N}}\left\{\frac{2\epsilon^2}{u_n^2+u_n}\right\}=\epsilon^2$,
$\{u_n\}$ is monotonically increasing and has a upper bound $1$. So 
$u_n\to u^*=1=\mathrm{sign} \left(u_0\right)$ as $n \rightarrow \infty$.
\end{proof}

\begin{rmk}
The result $\inf_{u_0\in \mathbb{R}}h^*(u_0,\epsilon)=0$ given in Theorem \ref{thm:CN} provides a reasonable explanation for Lemma \ref{lem:CN1}.
This explains why for any $h>0$, we can find $u_0$ that causes the numerical solution to converge to the wrong equilibrium state.
On the other hand, when $0<|u_0|<1$, we can observe that the maximum time step allowed to maintain the monotonicity of numerical solutions is $\epsilon^2$,
which is smaller than the maximum time step allowed by unique solvability of $2\epsilon^2$.
Therefore,  the maximum time step allowed by unique solvability is too large and it may not be sufficient to guarantee the accuracy of the numerical approximation. When $h\in(\epsilon^2,2\epsilon^2)$, which can disrupt the monotonicity of the numerical solution, leading to oscillatory behavior.  
\end{rmk}


It is proved in \cite{xu2023lack} that the Crank-Nicolson (CN) scheme is energy-stable with respect to a modified energy in the framework of energy minimization. 
We now show that  the numerical solution $\{u_n\}$ derived under the conditions of Theorem \ref{thm:CN} can maintain the stability of the original energy.

\begin{thm}\label{thm:CN2}
Given an initial condition $u_0$, when $h\in (0,h^*]$ and $h^*$ is given in Theorem \ref{thm:CN}, the numerical solution $u_n$ obtained by CN scheme $\eqref{eq:CN}$ maintain the original energy stability, namely $E(u_n)\leq E(u_{n-1})$.
\end{thm}

\begin{proof}
Multiplying $\eqref{eq:CN}$ by $\frac{1}{2}(u_n-u_{n-1})$ gives that
\begin{eqnarray}\label{eq:es-CN1}
\frac{(u_n-u_{n-1})^2}{2h}+\frac{1}{4\epsilon^2}\left[(u_n^3-u_n)(u_n-u_{n-1})+(u_{n-1}^3-u_{n-1})(u_n-u_{n-1})\right]=0,
\end{eqnarray}
which shows $\frac{1}{4\epsilon^2}\left[(u_n^3-u_n)(u_n-u_{n-1})+(u_{n-1}^3-u_{n-1})(u_n-u_{n-1})\right]\leq0$,
namely,
\begin{eqnarray}\label{eq:es-CN3}
E(u_n)- E(u_{n-1})+\frac{1}{4\epsilon^2}(u_n^2-u_{n-1}^2)(1-u_nu_{n-1})\leq0,
\end{eqnarray}
where $E(v)=\frac{1}{4\epsilon^2}(v^2-1)^2$. 
It follows from Theorem \ref{thm:CN} that for $h\in (0,h^*]$, $\{u_n\}$ maintains monotonicity and does not cross the equilibrium state. Hence, we have
\begin{eqnarray}\label{eq:es-CN4}
\frac{1}{4\epsilon^2}(u_n^2-u_{n-1}^2)(1-u_nu_{n-1})\geq 0.
\end{eqnarray}
Therefore, it follows  from \eqref{eq:es-CN3} that $E(u_n)\leq E(u_{n-1})$.
\end{proof}

\subsection{Modified Crank-Nicolson  scheme}

The second-order modified Crank-Nicolson   (modCN) scheme  for the ODE $(\ref{1.1})$ is given by

\begin{equation}
\frac{u_n-u_{n-1}}{h}+\frac{1}{\epsilon^2} \tilde{f}\left(u_n, u_{n-1}\right)=0, \quad n\geq 1,
\end{equation}
where
\begin{equation}
\tilde{f}\left(u, u_{n-1}\right)= \begin{cases}\frac{F(u)-F\left(u_{n-1}\right)}{u-u_{n-1}}, & u \neq u_{n-1}, \\ u^3-u, & u=u_{n-1} .\end{cases}
\end{equation}
Writing out $\tilde{f}\left(u, u_{n-1}\right)$, we obtain the following equivalent form:

\begin{equation}\label{eq:modCN1}
\frac{u_n-u_{n-1}}{h}+\frac{1}{\epsilon^2} \frac{u_n+u_{n-1}}{2}\left(\frac{\left(u_n\right)^2+\left(u_{n-1}\right)^2}{2}-1\right)=0 , \quad n\geq 1.
\end{equation}

\begin{lem}\label{lem:modCN1}
When $h \leq 2 \epsilon^2$, the modCN \eqref{eq:modCN1} scheme is uniquely solvable \cite{condette2011spectral}.  
When $h \leq 2 \epsilon^2$, it is also unconditionally energy-stable, i.e., $E(u_n)\leq E(u_n)+\frac{\left(u_n-u_{n-1}\right)^2}{h}=:E(u_{n-1}).$
However, for any step size $h>0$, there always exists an initial condition $u_0$ with $\left|u_0\right|>1$,
such that the modCN scheme converges to the wrong steady state solution \cite{xu2023lack}.
 \end{lem}

\begin{thm}
      Given an initial condition $u_0$, the modCN scheme \eqref{eq:modCN1}  for the
      ODE $(\ref{1.1})$ has the following properties.

(i) If $u_0\in \{0,\pm1\}$, when $h\leq 2\epsilon^2$, then $u_n = \operatorname{sign}(u_0)$  for all $n\geq1.$

(ii) If $u_0\notin \{0,\pm1\}$, then there exists a constant $h^*=h^*(u_0, \epsilon)>0$ such that for any $h\in(0, h^*]$, $\{u_n\}$ monotonically converge to correct equilibrium state $u^*=\mathrm{sign}(u_0) $ as $n\rightarrow\infty$. Specifically, we have $h^*=\frac{4\epsilon^2}{u_0^2+2|u_0|+1}$ for $|u_0|>1$ and $h^*=\epsilon^2$ for $0< |u_0|< 1$ respectively.
\end{thm}

\begin{proof}
The main method of proof is quite similar to Theorem \ref{thm:CN}, and we will omit its details here.
%
%
%
%
\end{proof}

\subsection{Implicit midpoint scheme}
The second-order implicit midpoint(IM) scheme for the ODE $(\ref{1.1})$ is given by
\begin{equation}\label{eq:IM}
\frac{u_n-u_{n-1}}{h}+\frac{1}{\epsilon^2}\left[f(\frac{u_n+u_{n-1}}{2})\right]=0, \, n\geq 1.
\end{equation}
 A related fully implicit scheme can be given by the following energy minimization formulation (IM-min)
\begin{equation}\label{eq1:IM}
u_n = \underset{v \in \mathbb{R}}{\arg\min} E_{IM}^n(v).
\end{equation}
where
\begin{equation}\label{eq2:IM}
E_{IM}^n(v)=\frac{(v-u_{n-1})^2}{4h}+\frac{1}{\epsilon^2}F(\frac{v+u_{n-1}}{2}).
\end{equation}
It is easy to see the derivative of the energy \eqref{eq2:IM} gives IM \eqref{eq:IM} and hence the
solution of \eqref{eq1:IM} is also a solution of \eqref{eq:IM}. The converse is not true except that $h\leq 2\epsilon^2$.
Therefore, when $h\leq 2\epsilon^2$, IM is uniquely solvable and is equivalent to IM-min

\begin{lem}\label{lem:IM1}
When  $h \leq 2 \epsilon^2$, there exists an initial value $u_0$ with $\left|u_0\right|>1$ such that the IM scheme converges to the wrong steady state solution.
\end{lem}

\begin{proof}
With a similar argument as that in lemma $\ref{lem:EE}$, we can first assume $u_1=-1$ or 1 and then solve $u_0>0$ or $u_0<0$ in reverse from equation \eqref{eq:IM}, which yields the required results.
\end{proof}

\begin{thm}\label{thm1:IM}
      Given an initial condition $u_0$, the IM scheme \eqref{eq:IM}  for the
      ODE $(\ref{1.1})$ has the following properties.

(i) If $u_0\in \{0,\pm1\}$, when $h\leq 2\epsilon^2$, then $u_n = \operatorname{sign}(u_0)$  for all $n\geq1.$

(ii) If $u_0\notin \{0,\pm1\}$, then there exists a constant $h^*=h^*(u_0, \epsilon)>0$ such that for any $h\in(0, h^*]$, $\{u_n\}$ monotonically converge to correct equilibrium state $\mathrm{sign}(u_0) $ as $n\rightarrow\infty$. Specifically, we have $h^*=\frac{8\epsilon^2}{u_0^2+4|u_0|+3}$ for $|u_0|>1$ and $h^*=\epsilon^2$ for $0< |u_0|< 1$ respectively.
\end{thm}

\begin{proof}
The main method of proof is also quite similar to Theorem \ref{thm:CN}, and we will omit its details here.
%
%
%
%
\end{proof}
At present, it is not known in the literature whether the standard IM scheme \eqref{eq:IM} maintains the stability of the original energy.
However, we can show that the numerical solution \( \{u_n\} \) obtained in Theorem \ref{thm1:IM} satisfies original energy stability.

\begin{thm}\label{thm:IM2}
      Given an initial condition $u_0$, when $h\in (0,h^*]$ and $h^*$ is given in Theorem \ref{thm1:IM}, the numerical solution $\{u_n\}$ obtained by IM scheme $\eqref{eq:IM}$  retains the energy stability, namely $E(u_n)\leq E(u_{n-1})$.
\end{thm}

\begin{proof}
Since $h^*< 2\epsilon^2$, the IM \eqref{eq:IM} is equivalent to IM-min \eqref{eq1:IM}, then $\{u_n\}$ obtained from \eqref{eq:IM} is also a global minimizer of \eqref{eq2:IM}. Therefore, we have $E_{IM}^n(u_n)\leq E_{IM}^n(u_{n-1})$. Thus,
\begin{equation}\label{eq3:IM}
\frac{(u_n-u_{n-1})^2}{4h}+\frac{1}{\epsilon^2}F(\frac{u_n+u_{n-1}}{2})\leq E(u_{n-1}),
\end{equation}
which shows
\begin{equation}\label{eq4:IM}
\frac{(u_n-u_{n-1})^2}{4h}+\frac{1}{\epsilon^2}\left(F(\frac{u_n+u_{n-1}}{2})-F(u_n)\right)\leq E(u_{n-1})- E(u_{n}).
\end{equation}
By the Lagrange's Mean Value Theorem, there exists $\xi\in\left(\min\{\frac{u_n+u_{n-1}}{2},u_n\},\max\{\frac{u_n+u_{n-1}}{2},u_n\}\right)$ such that
\begin{equation}\label{eq5:IM}
\frac{(u_n-u_{n-1})^2}{4h}+\frac{1}{\epsilon^2}(\xi^3-\xi)(\frac{u_{n-1}-u_{n}}{2})\leq E(u_{n-1})- E(u_{n}).
\end{equation}
It follow from Theorem \ref{thm1:IM} that for $h\in (0,h^*]$, $\{u_n\}$ maintains monotonicity and does not cross the equilibrium state, which yields that \begin{equation}\label{eq6:IM}
\frac{1}{\epsilon^2}(\xi^3-\xi)(\frac{u_{n-1}-u_{n}}{2})\geq0.
\end{equation}
Therefore, it follows from \eqref{eq5:IM} that  $E(u_{n})\leq E(u_{n-1})$.
\end{proof}

\subsection{Convex splitting of modCN scheme}

The second-order convex splitting (CS-modCN) scheme based
on the modCN scheme for the ODE $(\ref{1.1})$ is given by

\begin{eqnarray}\label{eq:cs1}
 \frac{u_n-u_{n-1}}{h}+\frac{1}{\epsilon^2}\left(\frac{u_n+u_{n-1}}{2}\right)\left(\frac{u_n^2+u_{n-1}^2}{2}\right)-\frac{1}{\epsilon^2}u_{n-1}=0, \quad n\geq 1.
\end{eqnarray}

\begin{lem}\label{lem1:cs}
(Ref \cite{xu2023lack})
For any $h >0$, the CS-modCN \eqref{eq:cs1} scheme is uniquely solvable and energy-stable.
However, there always exists an initial condition $u_0$ such that the CS-modCN scheme converges to the wrong steady state solution.
\end{lem}


\begin{thm}
      Given an initial condition $u_0$, the CS-modCN scheme \eqref{eq:cs1}  for the
      ODE $(\ref{1.1})$ has the following properties.

(i) If $u_0\in \{0,\pm1\}$, then $u_n = \mathrm{sign}(u_0)$ for all $n\geq1$ and for any $h$ and $\epsilon$;

(ii) If $u_0\notin \{0,\pm1\}$, then there exists a constant $h^*=h^*(u_0, \epsilon)>0$ such that for any $h\in(0, h^*]$, $\{u_n\}$ monotonically converge to correct equilibrium state $\mathrm{sign}(u_0) $ as $n\rightarrow\infty$. Specifically, we have $h^*=\frac{4\epsilon^2}{u_0^2+2|u_0|-1}$ for $|u_0|>1$ and $h^*=2\epsilon^2$ for $0< |u_0|< 1$ respectively.
\end{thm}

\begin{proof}
The main method of proof is also quite similar to Theorem \ref{thm:CN}, and we will omit its details here.
%
%
%
%
\end{proof}

\section{Numerical experiments}
\label{sec:num}

In this section, we conduct numerical experiments to validate the theoretical analysis and compare the numerical performance of different numerical methods.

$\textbf{Test 1.}$  First-order numerical schemes. 
We take $\epsilon=0.1$ and $u_0 = 0.5$ or $3$ with different time step sizes $h$.
 The critical value \(h^*\) is defined as the threshold that ensures the numerical solution to be monotonically and unique solvable. For the explicit Euler method, we have \(h^* = \frac{\epsilon^2}{u_0^2 + |u_0|}\) for \(|u_0| > 1\) and \(h^* = \frac{\epsilon^2}{2}\) when \(0 < |u_0| < 1\). In contrast, for the implicit Euler method, we have  \(h^* = \epsilon^2\). 
 
 The numerical results are presented in Figures \ref{Figure1} and \ref{Figure2}.
 We can observe from Figures \ref{Figure1} (left) that when the step size $h=0.005$, which equals to the threshold $h^*=0.005$, the numerical solution $u_{c}$ monotonically converges to the correct equilibrium state. However, when the step size $h=0.092$, which is slightly larger than our theoretical threshold $h^*$, rather than too much, the numerical solution $u_{o}$ exhibits oscillations and intersects with the equilibrium solution, ultimately converging to correct equilibrium state. 
 But when the step size $h=0.0015$, which exceeds the theoretical threshold $h^*=0.0008$ by a large amount,
 the numerical solution $u_{w}$ converges to an wrong equilibrium state, as shown in Figure \ref{Figure1} (right). 
 From Figure \ref{Figure2}, we find that as long as the implicit Euler is uniquely solvable, the exact solution can be effectively simulated.

\begin{figure}[H]\label{p1}
	\centering
	\subfloat{\label{Fig-1}
		\centering
		\includegraphics[width=0.45\textwidth]{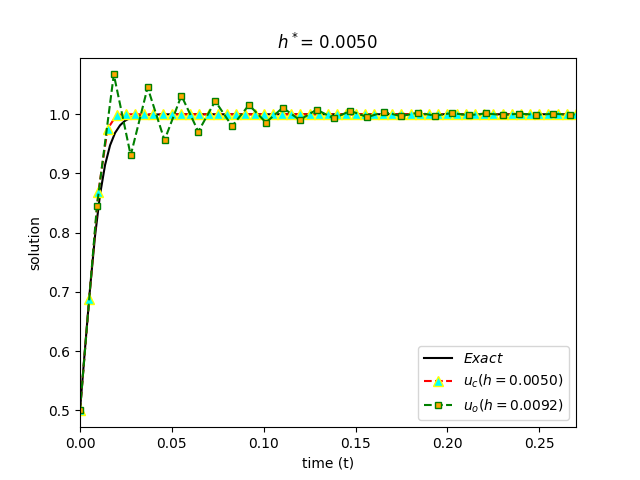}}
	\hspace{0.1cm}
	\subfloat{\label{Fig-2}
		\centering
		\includegraphics[width=0.45\textwidth]{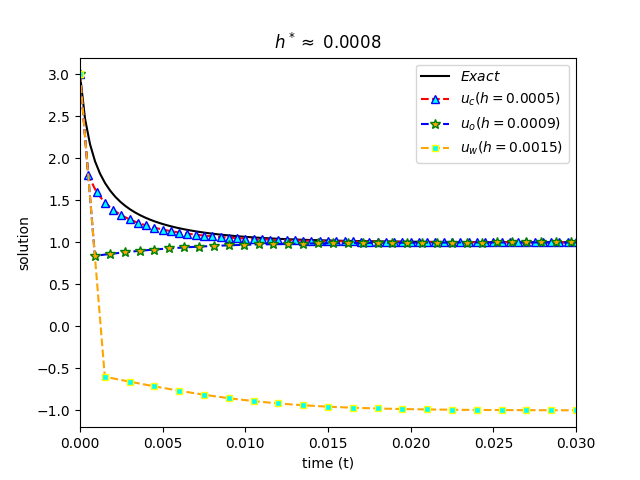}}
	\caption{Evolution of the numerical solution using EE scheme}
    \label{Figure1}
\end{figure}

\begin{figure}[H]\label{p2}
	\centering
	\subfloat{\label{Fig-3}
		\centering
		\includegraphics[width=0.45\textwidth]{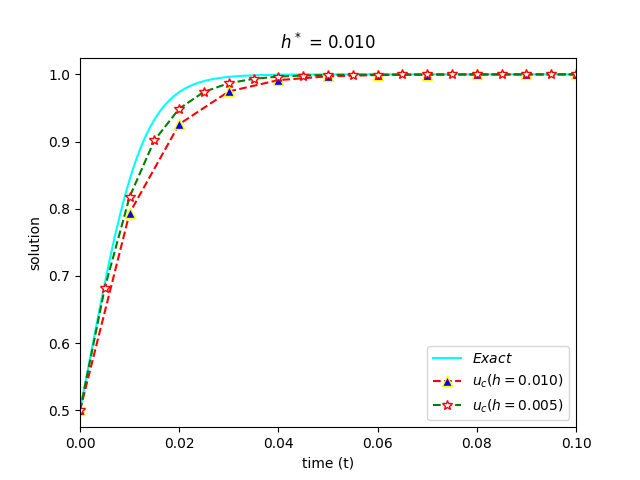}}
	\hspace{0.1cm}
	\subfloat{\label{Fig-4}
		\centering
		\includegraphics[width=0.45\textwidth]{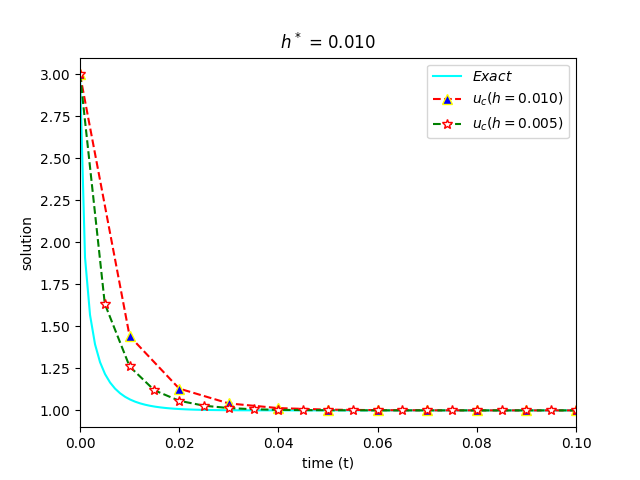}}
	\caption{Evolution of the numerical solution using IE scheme}
    \label{Figure2}
\end{figure}

$\textbf{Test 2.}$  Second-order numerical schemes. 
We take $\epsilon=0.1$ and $u_0 = 0.7$ or $3$ with different time step sizes $h$.
The theoretical threshold $h^*$ for various specific numerical methods are summarized in Table \ref{table1}.

%

  From Figures \ref{Figure3} to \ref{Figure6}, we can see that for various second-order algorithms discussed in this manuscript, when the step size $h$ is less than the corresponding theoretical threshold $h^*$, the numerical solution $u_{c}$ converges monotonically to the correct equilibrium state. When the initial value is less than $1$, and the step size exceeds the theoretical threshold but does not violate the condition for unique solvability, the numerical solution $u_{o}$ exhibits oscillations and intersects with the steady-state solution, eventually converging to the correct equilibrium state, as shown in Figures \ref{Figure3}(left)- \ref{Figure6}(left).  
Furthermore,  when the initial value is greater than $1$, as long as the actual simulation step size $h$ exceeds our theoretical threshold $h^*$, the numerical solution $u_{w}$ either oscillates  or converges to an wrong equilibrium state. 

It is worth noting that in some cases the numerical solutions oscillate only once and then monotonically converges to the correct equilibrium state is because $u_1$ falls within the interval $(0,1)$. At this point, the time step size corresponding to the initial value $u_0=3$ is less than the upper bound of the time step sizes corresponding to initial values $|u_0|<1$. This is shown in Figures \ref{Figure3}(right) - \ref{Figure6}(right).

\begin{figure}[H]
	\centering
	\subfloat{\label{Fig-5}
		\centering
		\includegraphics[width=0.45\textwidth]{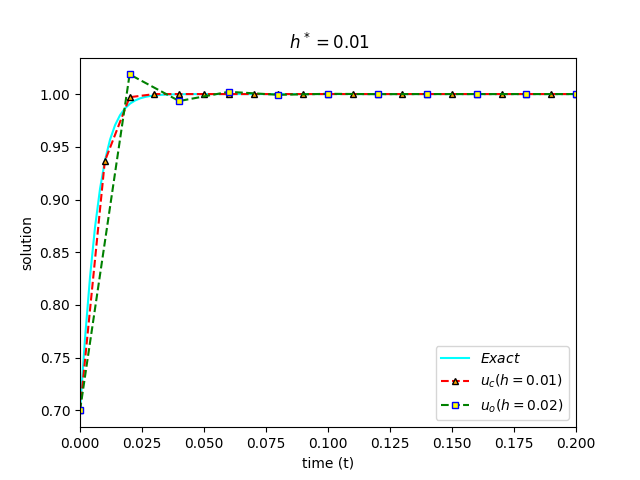}}
	\hspace{0.1cm}
	\subfloat{\label{Fig-6}
		\centering
		\includegraphics[width=0.45\textwidth]{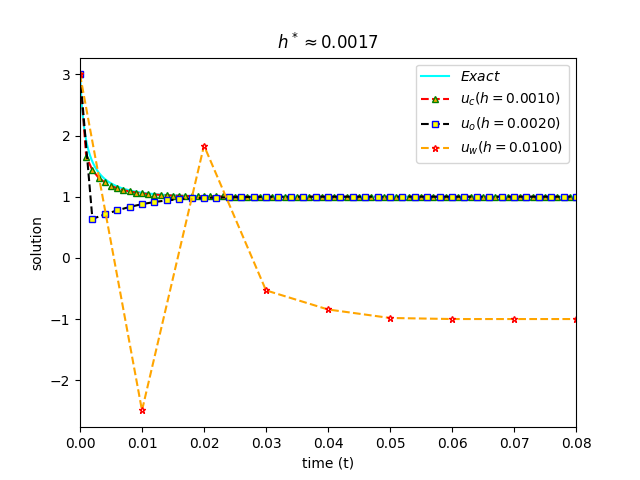}}
	\caption{Evolution of the numerical solution using CN scheme}
    \label{Figure3}
\end{figure}

\begin{figure}[H]
	\centering
	\subfloat{\label{Fig-7}
		\centering
		\includegraphics[width=0.45\textwidth]{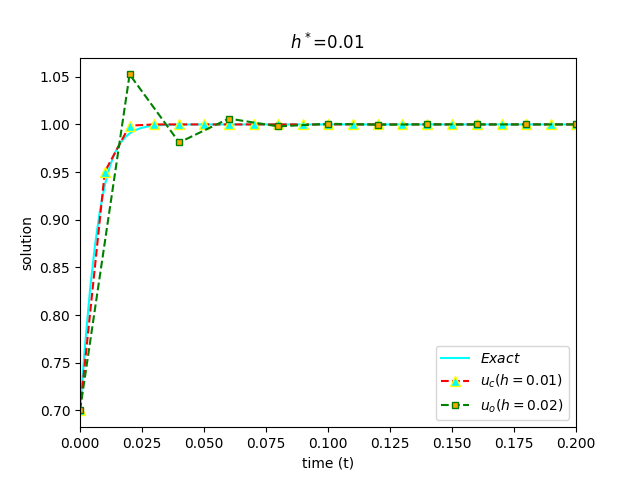}}
	\hspace{0.1cm}
	\subfloat{\label{Fig-8}
		\centering
		\includegraphics[width=0.45\textwidth]{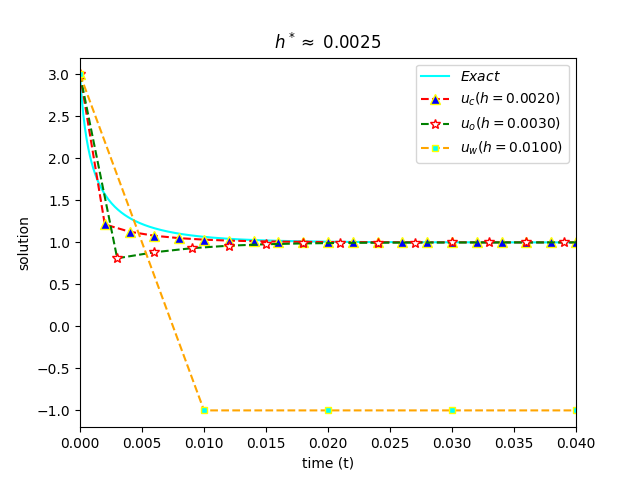}}
	\caption{Evolution of the numerical solution using modCN scheme}
    \label{Figure4}
\end{figure}

\begin{figure}[H]
	\centering
	\subfloat{\label{Fig-9}
		\centering
		\includegraphics[width=0.45\textwidth]{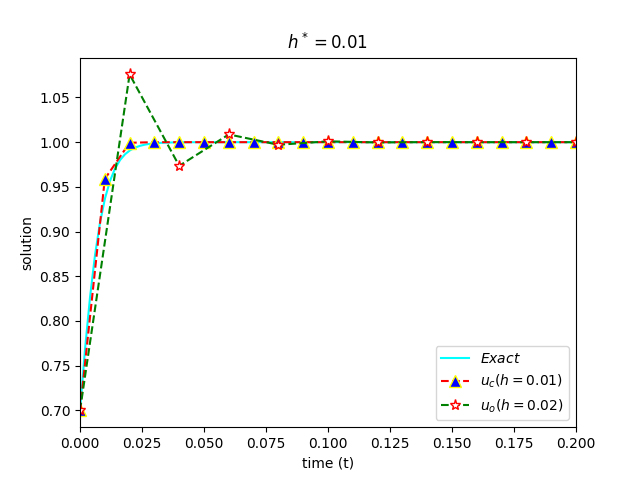}}
	\hspace{0.1cm}
	\subfloat{\label{Fig-10}
		\centering

		\includegraphics[width=0.45\textwidth]{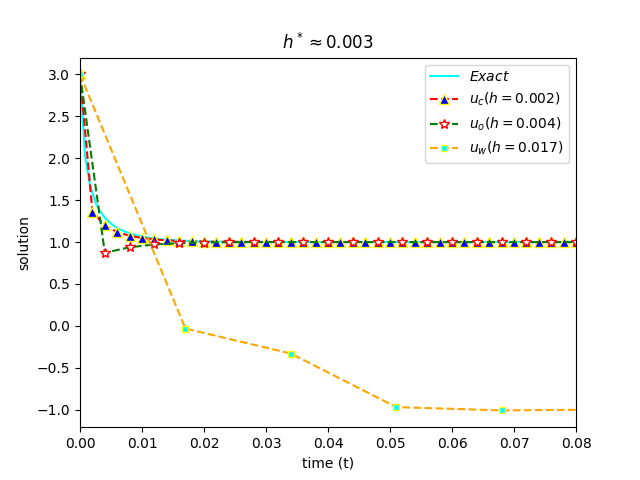}}
	\caption{Evolution of the numerical solution using IM scheme}
    \label{Figure5}
\end{figure}

\begin{figure}[H]
	\centering
	\subfloat{\label{Fig-11}
		\centering
		\includegraphics[width=0.45\textwidth]{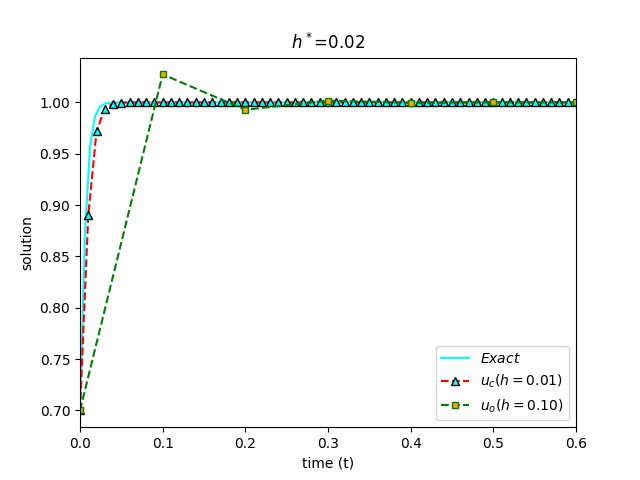}}
	\hspace{0.1cm}
	\subfloat{\label{Fig-12}
		\centering
		\includegraphics[width=0.45\textwidth]{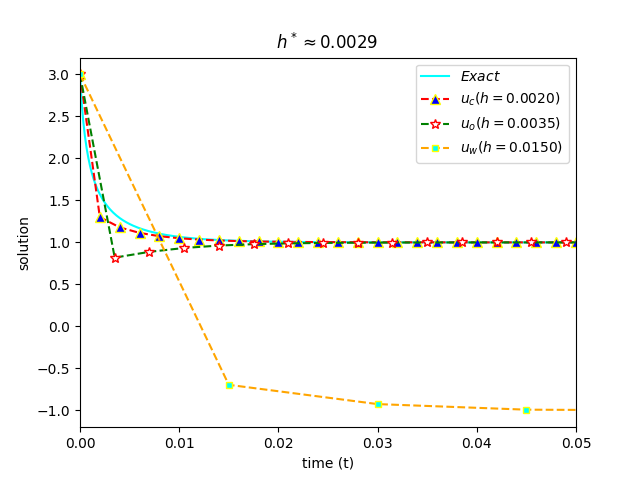}}
	\caption{Evolution of the numerical solution using CS-modCN scheme}
    \label{Figure6}
\end{figure}

\section{Concluding remarks}
\label{sec:con}
In this paper,  we introduce the concept of theoretical threshold $h^*=h^*(u_0, \epsilon)$ on the time steps for several numerical schemes that ensure the numerical solutions are monotonically and unique solvable, which is sufficient to ensure that many common numerical schemes converge to the correct equilibrium state.
The results are summarized and presented in Table \ref{table1}. Our results provide some supplement or explanation to the numerical analysis of the nonlinear phase field model in \cite{xu2023lack} from another perspective. 
In the next, we will consider extending the results of this note to nonlinear ODEs or Allan-Cahn equations.

\begin{table}[H]
    \centering
    \resizebox{1\textwidth}{!}{ 
        \begin{tabular}{cccccccc}
            \hline
            Numerical Scheme &  &  EE & IE & CN & modCN & CS-modCN & IM \\
            \hline
            \multirow{2}{*}{Monotonicity} & $|u_0|>1$ &$h\leq\frac{\epsilon^2}{u_0^2+|u_0|}$ & $h\leq\epsilon^2$&$h\leq\frac{2\epsilon^2}{u_0^2+|u_0|}$&$ h\leq\frac{4\epsilon^2}{u_0^2+2|u_0|+1} $& $h\leq\frac{4\epsilon^2}{u_0^2+2|u_0|-1}$&$h\leq\frac{8\epsilon^2}{u_0^2+4|u_0|+3}$\\
            \cline{2-8}
             & $0<|u_0|< 1$ & $h\leq \frac{\epsilon^2}{2}$ & $h\leq\epsilon^2$ & $h\leq\epsilon^2$ & $h\leq\epsilon^2$ & $h\leq 2\epsilon^2$ & $h\leq\epsilon^2$\\
            \hline
            \multirow{1}{*}{Unique solvability}&  &  $h<+\infty$ & $h\leq \epsilon^2$ & $h\leq 2\epsilon^2$ & $h\leq 2\epsilon^2$ & $h<+\infty$ & $h\leq 2\epsilon^2$\\
            \hline
            \multirow{2}{*}{$h^*$} & $|u_0|>1$ &$\frac{\epsilon^2}{u_0^2+|u_0|}$ & $\epsilon^2$&$\frac{2\epsilon^2}{u_0^2+|u_0|}$&$ \frac{4\epsilon^2}{u_0^2+2|u_0|+1} $& $\frac{4\epsilon^2}{u_0^2+2|u_0|-1}$&$\frac{8\epsilon^2}{u_0^2+4|u_0|+3}$\\
            \cline{2-8}
             & $0<|u_0|< 1$ &$ \frac{\epsilon^2}{2}$&$\epsilon^2$&$\epsilon^2$&$\epsilon^2$&$ 2\epsilon^2$&$\epsilon^2$\\
            \hline
        \end{tabular}
    }
    \caption{The main results of various numerical methods. The conditions for preserving unique solvability of numerical solutions, preserving monotonicity of numerical solutions, and the theoretical thresholds.} \label{table1}
\end{table}

\bibliographystyle{alpha}
\bibliography{NCH_ETD2}

\newcommand{\etalchar}[1]{$^{#1}$}
\begin{thebibliography}{GLWW14}

\bibitem[AC79]{allen1979microscopic}
S.~Allen and J.~Cahn.
\newblock A microscopic theory for antiphase boundary motion and its
  application to antiphase domain coarsening.
\newblock {\em Acta Metall}, 27(6):1085--1095, 1979.

\bibitem[AR99]{ascher1999midpoint}
U.~Ascher and S.~Reich.
\newblock The midpoint scheme and variants for hamiltonian systems: advantages
  and pitfalls.
\newblock {\em SIAM J. Sci. Comput}, 21(3):1045--1065, 1999.

\bibitem[CLKB21]{cheng2021asymptotic}
X.~Chen, D.~Li, P.~Keith, and W.~Brian.
\newblock Asymptotic behaviour of time stepping methods for phase field models.
\newblock {\em J. Sci. Comput}, 86(3):32, 2021.

\bibitem[CMS11]{condette2011spectral}
N.~Condette, C.~Melcher, and E.~S{\"u}li.
\newblock Spectral approximation of pattern-forming nonlinear evolution
  equations with double-well potentials of quadratic growth.
\newblock {\em Math. Comput}, 80(273):205--223, 2011.

\bibitem[CN47]{crank1947practical}
J.~Crank and P.~Nicolson.
\newblock A practical method for numerical evaluation of solutions of partial
  differential equations of the heat-conduction type.
\newblock {\em Math. proc}, 43(1):50--67, 1947.

\bibitem[DN91]{du1991numerical}
Q.~Du and R.~Nicolaides.
\newblock Numerical analysis of a continuum model of phase transition.
\newblock {\em SIAM J. Numer. Anal}, 28(5):1310--1322, 1991.

\bibitem[Eyr98]{eyre1998unconditionally}
D.~Eyre.
\newblock An unconditionally stable one-step scheme for gradient systems.
\newblock {\em Unpublished article}, 1998.

\bibitem[FP03]{feng2003numerical}
X.~Feng and A.~Prohl.
\newblock Numerical analysis of the allen-cahn equation and approximation for
  mean curvature flows.
\newblock {\em Numer. Math}, 94:33--65, 2003.

\bibitem[FTY13]{feng2013stabilized}
X.~Feng, T.~Tang, and J.~Yang.
\newblock Stabilized crank-nicolson adams-bashforth schemes for phase field
  models.
\newblock {\em East Asian J Appl Math}, 3(1):59--80, 2013.

\bibitem[GGT14]{guillen2014second}
F.~Guill{\'e}n-Gonz{\'a}lez and G.~Tierra.
\newblock Second order schemes and time-step adaptivity for allen--cahn and
  cahn--hilliard models.
\newblock {\em Comput Math Appl}, 68(8):821--846, 2014.

\bibitem[GKS13]{graser2013time}
C.~Gr{\"a}ser, R.~Kornhuber, and U.~Sack.
\newblock Time discretizations of anisotropic allen--cahn equations.
\newblock {\em IMA. J. Numer Anal}, 33(4):1226--1244, 2013.

\bibitem[GLWW14]{guan2014second}
Z.~Guan, J.~Lowengrub, C.~Wang, and S.~Wise.
\newblock Second order convex splitting schemes for periodic nonlocal
  cahn--hilliard and allen--cahn equations.
\newblock {\em J. Comput. Phys}, 277:48--71, 2014.

\bibitem[HK23]{ham2023stability}
S.~Ham and J.~Kim.
\newblock Stability analysis for a maximum principle preserving explicit scheme
  of the allen--cahn equation.
\newblock {\em Math Comput Simul}, 207:453--465, 2023.

\bibitem[HLT07]{he2007large}
Y.~He, Y.~Liu, and T.~Tang.
\newblock On large time-stepping methods for the cahn--hilliard equation.
\newblock {\em Appl Numer Math}, 57(5-7):616--628, 2007.

\bibitem[HLXX24]{hao2024stability}
W.~Hao, S.~Lee, X.~Xu, and Z.~Xu.
\newblock Stability and robustness of time-discretization schemes for the
  allen-cahn equation via bifurcation and perturbation analysis.
\newblock {\em arXiv preprint arXiv:2406.18393}, 2024.

\bibitem[HTY17]{hou2017numerical}
T.~Hou, T.~Tang, and J.~Yang.
\newblock Numerical analysis of fully discretized crank--nicolson scheme for
  fractional-in-space allen--cahn equations.
\newblock {\em J. Sci. Comput}, 72:1214--1231, 2017.

\bibitem[JLL{\etalchar{+}}16]{jeong2016comparison}
D.~Jeong, S.~Lee, D.~Lee, J.~Shin, and J.~Kim.
\newblock Comparison study of numerical methods for solving the allen--cahn
  equation.
\newblock {\em CMS}, 111:131--136, 2016.

\bibitem[LHY19]{li2019unconditionally}
C.~Li, Y.~Huang, and N.~Yi.
\newblock An unconditionally energy stable second order finite element method
  for solving the allen--cahn equation.
\newblock {\em J. Comput. Appl. Math}, 353:38--48, 2019.

\bibitem[SH98]{stuart1998dynamical}
A.~Stuart and A.~Humphries.
\newblock {\em Dynamical systems and numerical analysis}, volume~8.
\newblock Cambridge Univ. Press, 1998.

\bibitem[SXY18]{shen2018scalar}
J.~Shen, J.~Xu, and J.~Yang.
\newblock The scalar auxiliary variable (sav) approach for gradient flows.
\newblock {\em J. Comput. Phys}, 353:407--416, 2018.

\bibitem[SXY19]{shen2019new}
J.~Shen, J.~Xu, and J.~Yang.
\newblock A new class of efficient and robust energy stable schemes for
  gradient flows.
\newblock {\em SIAM Review}, 61(3):474--506, 2019.

\bibitem[SY10]{shen2010numerical}
J.~Shen and X.~Yang.
\newblock Numerical approximations of allen-cahn and cahn-hilliard equations.
\newblock {\em Discrete Contin. Dyn. Syst}, 28(4):1669--1691, 2010.

\bibitem[VdW79]{van1979thermodynamic}
J.~Van~der Waals.
\newblock The thermodynamic theory of capillarity under the hypothesis of a
  continuous variation of density.
\newblock {\em J Stat Phys}, 20(2):200--244, 1979.

\bibitem[WH16]{wang2016implicit}
P.~Wang and C.~Huang.
\newblock An implicit midpoint difference scheme for the fractional
  ginzburg--landau equation.
\newblock {\em J. Comput. Phys}, 312:31--49, 2016.

\bibitem[XLWB19]{xu2019stability}
J.~Xu, Y.~Li, S.~Wu, and B.~Bousquet.
\newblock On the stability and accuracy of partially and fully implicit schemes
  for phase field modeling.
\newblock {\em Comput. Method. Appl. M}, 345:826--853, 2019.

\bibitem[XT06]{xu2006stability}
C.~Xu and T.~Tang.
\newblock Stability analysis of large time-stepping methods for epitaxial
  growth models.
\newblock {\em SIAM J Numer Anal}, 44(4):1759--1779, 2006.

\bibitem[XX23]{xu2023lack}
J.~Xu and X.~Xu.
\newblock Lack of robustness and accuracy of many numerical schemes for
  phase-field simulations.
\newblock {\em Math. Mod. Meth. Appl. S}, 33(08):1721--1746, 2023.

\bibitem[Yan16]{yang2016linear}
X.~Yang.
\newblock Linear, first and second-order, unconditionally energy stable
  numerical schemes for the phase field model of homopolymer blends.
\newblock {\em J. Comput. Phys}, 327:294--316, 2016.

\bibitem[ZD09]{zhang2009numerical}
J.~Zhang and Q.~Du.
\newblock Numerical studies of discrete approximations to the allen--cahn
  equation in the sharp interface limit.
\newblock {\em SIAM J. Sci. Comput}, 31(4):3042--3063, 2009.

\bibitem[ZYQ{\etalchar{+}}21]{zhang2021preserving}
H.~Zhang, J.~Yan, X.~Qian, X.~Gu, and S.~Song.
\newblock On the preserving of the maximum principle and energy stability of
  high-order implicit-explicit runge-kutta schemes for the space-fractional
  allen-cahn equation.
\newblock {\em Numerical Algorithms}, 88:1309--1336, 2021.

\bibitem[ZYQS21]{zhang2021numerical}
H.~Zhang, J.~Yan, X.~Qian, and S.~Song.
\newblock Numerical analysis and applications of explicit high order maximum
  principle preserving integrating factor runge-kutta schemes for allen-cahn
  equation.
\newblock {\em Appl Numer Math}, 161:372--390, 2021.

\end{thebibliography}
\end{document}